\documentclass[11pt,reqno]{amsart}
\usepackage{amsmath,amsxtra,latexsym,amsthm,amssymb,amscd,pb-diagram}
\usepackage{mathrsfs,mathabx,amsfonts}
\usepackage[colorlinks=true,linkcolor=magenta,citecolor=blue]{hyperref}
\usepackage[margin=1in]{geometry}

\usepackage{bm}
\usepackage{color}
\usepackage{makecell,multirow,diagbox}
\usepackage{mathtools}
\usepackage[displaymath, mathlines]{lineno}

\usepackage{graphicx}
\usepackage{epsfig} 
\usepackage{array}
\usepackage{enumitem}

\newtheorem{definition}{Definition}[section]
\newtheorem{theorem}{Theorem}[section]
\newtheorem{proposition}{Proposition}[section]
\newtheorem{lemma}{Lemma}[section]

\newtheorem{remark}{Remark}[section]

\newcommand{\R}{\mathbb{R}}

\newcommand{\ity}{\infty}

\begin{document}
\title[semilinear damped wave equations]{On the asymptotic profile of solutions to semilinear damped wave equations with critical nonlinearities}

\thanks{$^* $\textit{Corresponding author:} Dinh Van Duong (vanmath2002@gmail.com)}

\maketitle
\centerline{\scshape  Trung Loc Tang$^{1}$, Dinh Van Duong$^{2,*}$ }
\medskip
{\footnotesize

{\footnotesize
	\centerline{$^1$ Department of Mathematics, Hanoi National University of Education}
	\centerline{No.136 Xuan Thuy road, Hanoi, Vietnam}}

	\centerline{$^2$Faculty of Mathematics and Informatics, Hanoi University of Science and Technology}
	\centerline{No.1 Dai Co Viet road, Hanoi, Vietnam}}

\begin{abstract}
    In this paper, we consider the Cauchy problem for a semilinear damped wave equation with the nonlinear term $|u|^{1+\frac{2}{n}}\mu(|u|)$, where $\mu$ is a modulus of continuity. In recent papers by Ebert–Girardi–Reissig \cite{EbeGirRei2020} and Girardi \cite{Girardi2024}, the authors obtained a sharp critical condition on $\mu$ in low space dimensions $n=1,2,3$, which determines the threshold between global (in time) existence of small data solutions and blow-up of solutions in finite time. Our new results are to prove that this condition remains valid in dimension $n=4$, together with the asymptotic profiles of global solutions. From this, we see that the behavior of the solution at $t \to \infty$ is identified by the Gauss kernel. Finally, a sharp lifespan estimate for local solutions is also derived in the case when blow-up occurs. 
\end{abstract}
% \linenumbers
\tableofcontents
\section{Introduction}
In this work, we consider the following Cauchy problem for semilinear damped wave equations:
\begin{equation} \label{Main.Eq.1}
\begin{cases}
u_{tt} -\Delta u + u_t= \mathcal{N}(u), &\quad x\in \R^n,\, t > 0, \\
(u, u_t)(0,x) = \varepsilon (u_0, u_1)(x), &\quad x\in \R^n, \\
\end{cases}
\end{equation}
where $\mathcal{N}(s) := |s|^{1+\frac{2}{n}}\mu(|s|)$ and the positive constant $\varepsilon$
describes the size of the initial data. The function $\mu : [0, \infty) \to [0, \infty)$ is a modulus of continuity, that is, $\mu$ is a continuous, concave, and increasing function satisfying $\mu(0) =0$. 

To begin with, we present the motivation for studying this model and review several well-known results related to it. The damped wave equation is a fundamental mathematical model describing wave motion when energy loss is present. Unlike the pure wave equation, it includes a damping term that accounts for friction, viscosity, or absorption in the medium. In the last half century, the following Cauchy problem has been extensively investigated by many mathematicians (see, for instance, \cite{Matsumura1976, Ikeda2019, Todorova2001, Zhang2001} and the references therein):
\begin{equation} \label{Main.Eq.2}
\begin{cases}
u_{tt} -\Delta u + u_t= |u|^p, &\quad x\in \R^n,\, t > 0, \\
(u, u_t)(0,x) = \varepsilon (u_0, u_1)(x), &\quad x\in \R^n, \\
\end{cases}
\end{equation}
where $p > 1$. In the classical paper \cite{Matsumura1976},
Matsumura was the first to consider the Cauchy problem of
nonlinear wave equations with dissipation terms. Their main tools are  some basic decay estimates established by  the
Fourier splitting method for the following linear damped wave equation:
\begin{equation} \label{Main.Eq.3}
\begin{cases}
u_{tt} -\Delta u + u_t= 0, &\quad x\in \R^n,\, t > 0, \\
(u, u_t)(0,x) = \varepsilon (u_0, u_1)(x), &\quad x\in \R^n. \\
\end{cases}
\end{equation}
In addition, they concluded that the damped wave equation has a diffusive structure as $t \to \infty$. This phenomenon has been further investigated in many subsequent papers, such as \cite{DabbiccoEbert2014, Michihisa2021, Takeda2015, Ikeda2019, ChenReissig2023} and the references therein. 
From this, they conclude that the diffusion phenomenon bridges the decay properties of solutions
to the Cauchy problem for the classical damped wave equation (\ref{Main.Eq.3}) and solutions to
the following Cauchy problem for the heat equation:
\begin{equation} \label{Main.Eq.4}
\begin{cases}
v_t -\Delta v = 0, &\quad x\in \R^n,\, t > 0, \\
v(0,x) = \varepsilon (u_0 + u_1)(x), &\quad x\in \R^n. \\
\end{cases}
\end{equation}
Specifically, they obtain the following  approximation of the solution $u$ by the Gauss kernel when $t$ is large:
\begin{align*}
    \mathcal{G}(t,x) := (4\pi t)^{-\frac{n}{2}} e^{-\frac{|x|^2}{4t}} = \mathfrak{F}^{-1}\big(e^{-|\xi|^2t}\big)(t,x).
\end{align*}
Furthermore, Michihisa \cite{Michihisa2021} discovered the higher-order asymptotic behavior of the solution to  (\ref{Main.Eq.3}) in the $L^2$ framework for any spatial dimension $n$ by using some tools, including Taylor series expansion and Faà di Bruno’s formula. Takeda \cite{Takeda2015} found the higher-order asymptotic behavior of solution to (\ref{Main.Eq.3})  in the $L^q$ framework for some $q\in (1,\ity)$ and concluded that  the effect of $u_{tt}$ in (\ref{Main.Eq.3}) is not
negligible for the third-order expansion. 

Taking into account the semilinear problem (\ref{Main.Eq.2})  under the additional regularity $L^1$ for the initial data, in \cite{IkehataMiyaokaNakatake2004, Todorova2001, IkehataOhta2002, IkehataTanizawa2005}, the authors showed that the Fujita exponent $p_F(n):= 1+2/n$ is the critical exponent of \eqref{Main.Eq.2} and the paper \cite{Zhang2001} showed that the value $p= p_F(n)$ belongs to the blow-up range.  Here, the critical exponent is understood as the threshold
between the global (in time) existence of small data solutions and the blow-up of solutions
even for small data. The Fujita exponent $p_F$
 is also the critical exponent for the semilinear problem (\ref{Main.Eq.4}), whose right-hand side is the nonlinear term $|v|^p$. Specifically, Ikehata et al. \cite{IkehataMiyaokaNakatake2004}  succeeded in proving the global existence and optimal decay
estimates of the total energy of the weak solutions to problem (\ref{Main.Eq.2}) with the power $p > p_F(n)$ for $n =1,2$, and $p > 2$ for $n \geq 3$. It should be mentioned that the results in \cite{Todorova2001}  fully depend on
the compactness assumption on the support of initial data, while in \cite{IkehataMiyaokaNakatake2004} these were removed. Next, Gallay-Raugel \cite{GallayRaugel1998} proved that global solutions of nonlinear damped wave equation
behaves like those of nonlinear heat equations with suitable data, including
more general nonlinearity for $n=1$. Karch \cite{Karch2000} proved the approximation
of the solution to (\ref{Main.Eq.2}) by the Gauss kernel for $p \geq 1 +4/n$. Nishihara \cite{Nishihara2003}  proved it for $p > 1+2/n$ when $n=3$, followed by \cite{MarcatiNishihara2003} for $n=1$, \cite{HosonoOgawa2004} for $n=2$, \cite{Narazaki20242} for $n=4, 5$ and \cite{Hayashi2004} for all $n \geq 1$. More gererally, Kawakami-Takeda \cite{KawakamiTakeda2016}  established the higher order asymptotic expansion
of the solution to (\ref{Main.Eq.2}) under suitable assumptions for the nonlinearity and the initial
data. For the blow-up range $p \leq 1 + 2/n$, according to the works \cite{LiZhou1995, Ikeda2016, LaiZhou2019}, the sharp lifespan estimates for blow-up solutions to (\ref{Main.Eq.2}) in all spatial dimensions have been investigated. Here, we denote by $T_{\varepsilon}$ the lifespan of solution in the following sense:
$$
\begin{aligned}
T_{\varepsilon}:=\sup\{& T > 0:\text{ there exists a unique local (in time) solution $u$ on $[0,T)$}\\
&\text{ with a fixed parameter $\varepsilon>0$}\}.
\end{aligned}
$$ 
Consequently, these papers provided the following sharp lifespan estimates:
\begin{align*}
    T_{\varepsilon} \sim \varepsilon^{-\frac{2(p-1)}{2-n(p-1)}} \text{ if } 1 < p < 1+\displaystyle\frac{2}{n} \quad\text{ and }\quad \log(T_{\varepsilon}) \sim  \varepsilon^{-\frac{2}{n}} \text{ if } p = 1+\displaystyle\frac{2}{n}.
\end{align*}
From the above discussion, one has the Fujita exponent $p_F= 1+2/n$ serves as the threshold separating the blow-up range from the global existence range of global (in time) solutions to problem (\ref{Main.Eq.2}) with small data. Moreover, the critical case $p = p_F$ belongs to the blow-up range. Recently, Ebert et al. \cite{EbeGirRei2020} investigated the Cauchy problem \eqref{Main.Eq.1} with nonlinear term $|u|^{1+\frac{2}{n}} \mu(|u|)$ and established the sharp condition on $\mu(s)$ that separates the blow-up case from the global existence case. Specifically, the problem (\ref{Main.Eq.1}) has a unique global (in time) small data solution if $n=1,2$ and the \textit{Dini condition} is satisfied. Moreover, problem (\ref{Main.Eq.1}) admits a blow-up solution in finite time for all spatial dimensions $n \geq 1$ if $\mu$ satisfies the \textit{non-Dini condition}.
\begin{definition}
    Let $\mu : [0, +\infty) \to [0, +\infty)$ be a modulus of continuity. Then, $\mu$ satisfies the Dini condition if
    \begin{align}
        \int_0^{1} \frac{\mu(s)}{s} ds < +\infty. \label{Condition1.1.1}
    \end{align}
    On the other hand, $\mu$ satisfies the non-Dini condition if (\ref{Condition1.1.1}) does not hold, i.e., if
    \begin{align}
        \int_0^{1} \frac{\mu(s)}{s} ds = +\infty. \label{Condition1.2.1}
    \end{align}
\end{definition}
 Extending the work \cite{EbeGirRei2020}, Girardi \cite{Girardi2024} showed that the results obtained in \cite{EbeGirRei2020} can be generalized to
more general semilinear evolution models with the Fujita-type critical exponent; in particular,
the critical behavior of nonlinearity is still described by the Dini condition (\ref{Condition1.1.1}). From this, they extended the global existence result of (\ref{Main.Eq.1}) to the case $n=3$ by using the estimates of Nishihara in \cite{Nishihara2003}. Continuing with this topic, Dao-Reissig \cite{AnhRei2021} studied for semilinear damped wave systems, while the papers of Chen-Girardi \cite{ChenGirardi2025} and Dao-Son \cite{DaoSon2025} investigated it in the framework of evolution equations. They also proved that, under the condition (\ref{Condition1.1.1}), these problems admit a unique global (in time) solution for small data. However, their assumptions on the spatial dimension are rather restrictive, since they are constrained by tools from harmonic analysis. In addition, the readers may refer to paper \cite{ChenReissig2024} for the semilinear classical wave equations. Our first main goal in this paper extends the global existence property of (\ref{Main.Eq.1}) to dimensions $1 \leq n \leq 4$ under condition (\ref{Condition1.1.1}) by applying the linear estimates obtained in \cite{Ikeda2019}. Next, our second main objective is to determine the asymptotic behavior of global (in time) solutions to (\ref{Main.Eq.1}) in spatial dimensions $1 \leq n \leq 4$. More specifically, we show that the behavior of the solution at $t \to \infty$ is identified by the Gauss kernel under the Dini condition (\ref{Condition1.1.1}) for the global existence 
\begin{align*}
    \lim_{t \to \infty} t^{\frac{n}{2}(1-\frac{1}{q})}\left\|u(t,\cdot)- M \mathfrak{F}^{-1}\big(e^{-|\xi|^2t}\big)(t,\cdot) \right\|_{L^q} = 0,
\end{align*}
for all $q \in \left[\min\{2, 1+2/n\}, +\infty\right]$ and $M$ defined by (\ref{Defi_1}). Finally, we will provide some remarks on sharp lifespan estimates for blow-up solutions to problem (\ref{Main.Eq.1}) under the non-Dini condition (\ref{Condition1.2.1}).
\vspace{0.3cm}

\textbf{Notations.} 
\begin{itemize}[leftmargin=*]
    \item We write $f\lesssim g$ when there exists a constant $C>0$ such that $f\leq Cg$, and $f \sim g$ when $g\lesssim f\lesssim g$. 

\item For any $\eta \in \mathbb{R}$, we denote by $[\eta]^+ := \max\{0, \eta\}$, its positive part.
    
    \item In addition, we denote $\widehat{w}(t,\xi):= \mathfrak{F}_{x\rightarrow \xi}\big(w(t,x)\big)$ as the Fourier transform with respect to the spatial variable of a function $w(t,x)$ and $\mathfrak{F}^{-1}$ represents the inverse Fourier transform. 
    \item As usual, $H^{a}_m$ and $\dot{H}^{a}_m$, with $m \in (1, \infty), a \geq 0$, denote potential spaces based on $L^m$ spaces. Here $\langle \nabla\rangle^{a}$ and $|\nabla|^{a}$ stand for the pseudo-differential operators with symbols $\big<\xi\big>^{a}$ and $|\xi|^{a}$, respectively, where the symbol $\langle x\rangle := \sqrt{|x|^2 +1} $ denotes the Japanese bracket.
\end{itemize}
%====================================================================================
%=================================================================================={Introduction}	

\textbf{Main results.} Let us state the global (in time) existence of small data solutions, together with their asymptotic behavior, which will be proved in this paper.
\vspace{0.2cm}

\begin{theorem}[\textbf{Global existence}]\label{Theorem1}
     Let $1 \leq n \leq 4$. The modulus of continuity $\mu(s) $ satisfies the Dini condition (\ref{Condition1.1.1})
     and
     \begin{align}\label{Condition1.1.2}
         s|\mu'(s)| \lesssim |\mu(s)|  \quad \text{ for } s \in (0, 1].
     \end{align}
    In addition, we fix $$\alpha := \min\left\{2, 1+\frac{2}{n}\right\}, \,\,\beta_\alpha := (n-1)\left(\frac{1}{\alpha}-\frac{1}{2}\right)$$ and assume that
    \begin{align*}
        r \in \left(2, \frac{2n}{[n-2]^+}\right).
    \end{align*}
    Furthermore, the initial data $(u_0, u_1)$ satisfies
    \begin{align*}
        (u_0, u_1) \in \mathcal{D} := \left(H^2 \cap H^{2}_r \cap H^{\beta_\alpha}_\alpha\cap L^1  \right)\times \left(H^1 \cap H^1_r \cap L^{\alpha} \cap L^1\right).
    \end{align*}
    Then, there exists a constant $\bar{\varepsilon} > 0$ such that for any $\varepsilon \in (0, \bar{\varepsilon}]$, problem (\ref{Main.Eq.1}) admits a unique global (in time) Sobolev solution
    \begin{align*}
        u \in \mathcal{C}([0, \infty), H^2 \cap L^\alpha \cap L^\infty)
    \end{align*}
    satisfying the following estimates:
    \begin{align*}
        \|u(t,\cdot)\|_{\dot{H}^2} &\lesssim \varepsilon(1+t)^{-\frac{n}{4}-1} \|(u_0,u_1)\|_{\mathcal{D}},\\
        \|u(t,\cdot)\|_{L^{\alpha}} &\lesssim \varepsilon(1+t)^{-\frac{n}{2}(1-\frac{1}{\alpha})} \|(u_0,u_1)\|_{\mathcal{D}},\\
        \|u(t,\cdot)\|_{L^\infty} &\lesssim \varepsilon(1+t)^{-\frac{n}{2}}  \|(u_0,u_1)\|_{\mathcal{D}}.
    \end{align*}
\end{theorem}

\begin{remark}
\fontshape{n}
\selectfont
    From the statements of Theorem \ref{Theorem1}, we assert that the Dini condition (\ref{Condition1.1.1}) remains indispensable for the global (in time) existence of small data solutions to problem (\ref{Main.Eq.1}) in $4$-dimension space. This result extends Theorem 3 in \cite{EbeGirRei2020} and Proposition 4.1  in \cite{Girardi2024}.  For some examples of modulus of continuity $\mu$ satisfying the conditions of Theorem \ref{Theorem1}, the reader may refer to Example 1 in \cite{EbeGirRei2020}. 
\end{remark}

\begin{remark}
\fontshape{n}
\selectfont
    The approach used to prove Theorem \ref{Theorem1} can still be effectively applied to the weakly coupled system of equations \eqref{Main.Eq.1}. Namely, we can extend the global existence result in \cite{AnhRei2021} to higher dimensions $n=3,4$.
\end{remark}

\begin{theorem}[\textbf{Asymptotic profiles}] \label{Theorem3}
    Assume that the assumptions of Theorem \ref{Theorem1} hold, the global (in time)
small data solution to (\ref{Main.Eq.1}) satisfies the following estimates for $t \gg 1$:
\begin{align*}
    \|u(t,\cdot) - M \mathcal{G}(t,\cdot)\|_{L^\alpha} &= o(t^{-\frac{n}{2}(1-\frac{1}{\alpha})}),\\
    \|u(t,\cdot)-M \mathcal{G}(t,\cdot)\|_{\dot{H}^2} &= o(t^{-\frac{n}{4}-1}),\\
    \|u(t,\cdot)-M \mathcal{G}(t,\cdot)\|_{L^\infty} &= o(t^{-\frac{n}{2}}),
\end{align*}
where
\begin{align}\label{Defi_1}
    M := \varepsilon\int_{\mathbb{R}^n} (u_0(x) +  u_1(x)) dx + \int_0^{\infty} \int_{\mathbb{R}^n} \mathcal{N}(u(\tau,x)) dx d\tau.
\end{align}
\end{theorem}
\begin{remark}
\fontshape{n}
\selectfont
    From the statements of Theorem \ref{Theorem3}, we  see that the asymptotic behavior of the global solution at $t \to \infty$ is identified by the Gauss kernel under the Dini condition (\ref{Condition1.1.1}) as follows:
\begin{align*}
    u(t,\cdot) \sim \left(\varepsilon\int_{\mathbb{R}^n} (u_0(x) +  u_1(x)) dx + \int_0^{\infty} \int_{\mathbb{R}^n} \mathcal{N}(u(\tau,x)) dx d\tau\right) \mathfrak{F}^{-1}\big(e^{-|\xi|^2t}\big)(t,\cdot),  
\end{align*}
in the $L^q$ framework, for all $q \in \left[\min\{2, 1+2/n\}, +\infty\right]$.  From this, we can conclude the diffusive structure of problem (\ref{Main.Eq.1}) when it admits a unique global solution because it presents the following optimal decay estimates:
\begin{align*}
    \|u(t,\cdot)\|_{L^q} \sim t^{-\frac{n}{2}(1-\frac{1}{q})} \quad\text{ and } \quad
    \|u(t,\cdot)\|_{\dot{H}^2} \sim t^{-\frac{n}{4}-1},
\end{align*}
where $t \gg 1$. These estimates can be derived from the conclusion of Theorem \ref{Theorem3} together with the triangle inequality.
\end{remark}
\textbf{This paper is organized as follows:} In Section \ref{Section2}, we provide the proof of global (in time) existence results for solutions to the problem (\ref{Main.Eq.1}). Subsequently, in Section \ref{Section_Asym}, we describe the asymptotic behavior of the global solution to (\ref{Main.Eq.1}). Finally, we provide some remarks on sharp lifespan estimates in Section \ref{Lifespan}.

\section{Global Existence}\label{Section2}
\subsection{Linear estimates and solution spaces} 
To begin with, we can write the solution to $(\ref{Main.Eq.3})$ using the formula
\begin{equation*}
 u^{\rm lin}(t,x) = \varepsilon (\mathcal{K}(t,x) +\partial_t \mathcal{K}(t,x))\ast_x u_0(x) + \varepsilon \mathcal{K}(t,x) \ast_x u_1(x),
\end{equation*}
so that the solution to $(\ref{Main.Eq.1})$ becomes
\begin{align}
    u(t,x)
    = u^{\rm lin}(t,x) + u^{\rm non}(t,x), \label{Solution}
\end{align}
thanks to Duhamel's principle, where
\begin{align*}
    u^{\rm non}(t,x) := \int_0^t \mathcal{K}(t-\tau,x) \ast_x \mathcal{N}(u(\tau,x)) d\tau
\end{align*}
and the linear kernel $\mathcal{K}(t,x)$ is defined by
\begin{align*}
    \mathcal{K}(t,x) :=  \begin{cases}
        \vspace{0.3cm}\mathfrak{F}^{-1}\left(\displaystyle\frac{e^{-\frac{t}{2}}\sinh{\left(t \sqrt{\frac{1}{4} -|\xi|^2}\right)}}{\sqrt{\frac{1}{4}- |\xi|^2}} \right)(t,x) &\text{ if } |\xi| \leq \displaystyle\frac{1}{2},\\
        \displaystyle\mathfrak{F}^{-1}\left(\frac{e^{-\frac{t}{2}}\sin{\left(t \sqrt{|\xi|^2-\frac{1}{4}}\right)}}{\sqrt{|\xi|^2-\frac{1}{4}}}\right)(t,x) &\text{ if } |\xi| > \displaystyle\frac{1}{2}.
    \end{cases}
\end{align*}
Let $\chi_k= \chi_k(r)$ with $k\in\{\rm L,H\}$ be smooth cut-off functions having the following properties:
\begin{align*}
&\chi_{\rm L}(r)=
\begin{cases}
1 &\quad \text{ if }r\le \varepsilon^*/2, \\
0 &\quad \text{ if }r\ge \varepsilon^*,
\end{cases}
\text{ and } \qquad
\chi_{\rm H}(r)= 1 -\chi_{\rm L}(r),
\end{align*}
where $\varepsilon^*$ is a sufficiently small constant.
It is obvious to see that $\chi_{\rm H}(r)= 1$ if $r \geq \varepsilon^*$ and $\chi_{\rm H}(r)= 0$ if $r \le \varepsilon^*/2$.\\
Now, we restate the following important result.

\begin{lemma}[\textbf{Linear Estimates}]\label{LinearEstimates}
    Let $n \geq 1,\, j \in \{ 0,1\}, \,
    1 \leq \rho \leq q < \infty$, $q \ne 1$, $\beta_q := (n-1)\left|\frac{1}{2}-\frac{1}{q}\right|$ and $s_1 \geq s_2 \geq 0$.  Then, the following estimate holds for all $t > 0$:
    \begin{align*}
        &\|\partial_t^j |\nabla|^{s_1} \mathcal{K}(t,x)\ast_x \varphi(x)\|_{L^q}\\ &\quad\lesssim (1+t)^{-\frac{n}{2}(\frac{1}{\rho}-\frac{1}{q})-\frac{s_1-s_2}{2}-j} \||\nabla|^{s_2} \chi_{\rm L}(|\nabla|)\varphi\|_{L^{\rho}} + e^{-ct} \||\nabla|^{s_1}\chi_{\rm H}(|\nabla|) \varphi\|_{H_{q}^{\beta_q+j-1}},
        \end{align*}
        where $c$ is a suitable positive constant.
    Furthermore, we obtain the following estimate for $ 1\leq \rho \leq \infty, 1 < r < \infty$ and $d > \displaystyle\frac{n}{r}$:
    \begin{align*}
        &\|\partial_t^j |\nabla|^{s_1}\mathcal{K}(t,x) \ast_x \varphi(x) \|_{L^{\infty}}\\ &\quad\lesssim (1+t)^{-\frac{n}{2\rho}-\frac{s_1-s_2}{2}-j} \||\nabla|^{s_2} \chi_{\rm L}(|\nabla|) \varphi\|_{L^\rho} + e^{-ct} \||\nabla|^{s_1 + d} \chi_{\rm H}(|\nabla|) \varphi\|_{H^{\beta_r+j -1}_r}.
    \end{align*}
\end{lemma}

\begin{proof}
   Due to the similarity in the proof, we present only the proof of the case $j=0$. Specifically, using Proposition 2.4 in \cite{Ikeda2019}, we gain
\begin{align*}
    &\||\nabla|^{s_1} \chi_{\rm L}(|\nabla|) \mathcal{K}(t,x) \ast_x \varphi(x)\|_{L^q} \\
    &\quad\lesssim (1+t)^{-\frac{n}{2}(\frac{1}{\rho}-\frac{1}{q})-\frac{s_1-s_2}{2}} \||\nabla|^{s_2} \chi_{\rm L}(|\nabla|) \varphi\|_{L^\rho}, \text{ for all } 1 \leq \rho \leq q \leq \infty.
\end{align*}
    In addition, from Proposition 2.5 in \cite{Ikeda2019}, we obtain 
    \begin{align*}
         \| |\nabla|^{s_1} \chi_{\rm H}(|\nabla|) \mathcal{K}(t,x)\ast_x \varphi(x)\|_{L^q} \lesssim e^{-ct} \||\nabla|^{s_1} \chi_{\rm H}(|\nabla|) \varphi\|_{H^{\beta_q-1}_q} \text{ for } 1 < q < \infty.
    \end{align*}
    Moreover, we can see that
    \begin{align}
        &\||\nabla|^{s_1} \chi_{\rm H}(|\nabla|) \mathcal{K}(t,x) \ast_x \varphi(x)\|_{L^{\infty}} \notag\\
        &\quad\lesssim \||\nabla|^{s_1} \langle \nabla \rangle^{d} \chi_{\rm H}(|\nabla|) \mathcal{K}(t,x) \ast_x\varphi(x)\|_{L^r} \lesssim e^{-ct} \||\nabla|^{s_1+d} \chi_{\rm H}(|\nabla|) \varphi\|_{H^{\beta_r-1}_r}, \label{Esti100}
    \end{align}
    for all $1< r < \infty$ and $d > n/r$.
    These complete the proof of case $j=0$.
\end{proof}

Under the assumptions of Theorem \ref{Theorem1},  we define the following function spaces for $T>0$
\begin{align*}
    X(T) :=  L^{\infty}([0, T], H^2 \cap L^{\alpha} \cap L^{\infty}),
\end{align*}
with the norm
\begin{align*}
    &\|\varphi\|_{X(T)} \\
    &\,\,:= \sup _{t \in [0,T]} \bigg\{(1+t)^{\frac{n}{2}(1-\frac{1}{\alpha})}\|\varphi(t,\cdot)\|_{L^{\alpha}} + (1+t)^{\frac{n}{4}+1} \| \varphi(t,\cdot)\|_{\dot{H}^2} + (1+t)^{\frac{n}{2}} \|\varphi(t,\cdot)\|_{L^\infty} \bigg\}
\end{align*}
and 
\begin{align*}
    Y(T) := L^{\infty}([0,T], L^2),
\end{align*}
with the norm
\begin{align*}
    \|\varphi\|_{Y(T)} := \sup_{t \in [0,T]
    }\big\{\|\varphi(t,\cdot)\|_{L^2} \big\}.
\end{align*}
Moreover, we denote that 
\begin{align*}
    X(T, M) := \{\varphi \in X(T): \|\varphi\|_{X(T)} \leq M \},
\end{align*}
for all $M > 0$. We now proceed to the following important property.
\vspace{0.2cm}

\begin{lemma}\label{lemma1.2}
    $X(T, M)$  is a closed subset of $Y(T)$ with respect to the metric $Y(T)$. 
\end{lemma}
\begin{proof}
Firstly, we can see that $X(T, M) \subset Y(T)$ by interpolation. Therefore, it suffices to show that
if a sequence in $X(T, M)$ converging in $Y(T)$, its limit will belong to $X(T, M)$. More specifically, we assume that $\{\varphi_j\}_{j=1}^{\infty} \subset X(T,M)$ and $\varphi_j \to \varphi \in Y(T)$ as $j \to \infty$. One has the following relation:
    \begin{align*}
        L^\infty\left([0,T], H^2 \cap L^{\infty} \cap L^\alpha \right) = \left(L^1([0,T], H^{-2} + L^1 + L^{\alpha'})\right)^*,
    \end{align*}
    where $\alpha' = \alpha/(\alpha-1)$. Due to the separability of $L^1([0,T], H^{-2} + L^1 + L^{\alpha'})$, we apply Banach-Alaoglu theorem (see Theorem 3.16 or Corollary 3.30 in \cite{Brezis2011}) to take a subsequence $\{\varphi_{j(k)}\}_{k=1}^\infty$ and $\psi \in L^\infty\left([0,T], H^2 \cap L^{\infty} \cap L^\alpha \right) $ such that
    \begin{align*}
        \varphi_{j(k)} \overset{*}{\rightharpoonup} \psi \quad\text{ as } \quad k \to \infty .
    \end{align*}
    In addition, we have 
    \begin{align*}
        \|\psi\|_{X(T)} \leq \liminf_{k\to \infty} \|\varphi_{j(k)}\|_{X(T)} \leq M, 
    \end{align*}
    that is, $\psi \in X(T, M)$.
    On the other hand, both $\{\varphi_j\}_{j=1}^\infty$
and $\{\varphi_{j(k)}\}_{k=1}^{\infty}$
converge in the space of the distribution $
\mathcal{D}'([0, T] \times \mathbb{R}^n)$, that is,
\begin{align*}
    \varphi_{j(k)} &\to \psi \in \mathcal{D}'([0, T] \times \mathbb{R}^n) \quad\text{ as }\quad k \to \infty,\\
    \varphi_{j} &\to \varphi \in \mathcal{D}'([0, T] \times \mathbb{R}^n) \quad\text{ as } \quad j \to \infty.
\end{align*}
As a result, the uniqueness of the limit of distribution implies $\psi \equiv \varphi, $ which shows $\varphi \in X(T,M)$. 
\end{proof}

Next, we consider the operator $\Phi$ on the space $X(T)$
\begin{align}\label{Mapping1}
   \Phi[u] := u^{\rm lin} + u^{\rm non}
\end{align}
and denote the quantity
\begin{align*}
    \mathcal{I}(\varepsilon_0) := \int_0^{\varepsilon_0} \frac{\mu(s)}{s} ds,
\end{align*}
for $\varepsilon_0 \in (0, 1]$. Due to the condition (\ref{Condition1.1.1}), we have $\mathcal{I}(\varepsilon_0) \to 0$ as $\varepsilon_0 \to 0^+$. Finally, we introduce an important tool from Harmonic Analysis.
\vspace{0.2cm}

\begin{proposition}[see Corollary 2.4 in \cite{Hajaiej2011}]  \label{fractionalGagliardoNirenberg}
Let $1<p,\,p_0,\,p_1<\infty$, $a >0$ and $\sigma\in [0,a)$. Then, we have the following fractional Gagliardo-Nirenberg inequality:
$$ \|u\|_{\dot{H}^{\sigma}_p} \lesssim \|u\|_{L^{p_0}}^{1-\omega(\sigma,a)}\, \|u\|_{\dot{H}^{a}_{p_1}}^{\omega(\sigma,a)}, $$
where $\omega(\sigma,a) :=\displaystyle\frac{\frac{1}{p_0}-\frac{1}{p}+\frac{\sigma}{n}}{\frac{1}{p_0}-\frac{1}{p_1}+\frac{a}{n}}$ and $\displaystyle\frac{\sigma}{a}\leq \omega(\sigma,a) \leq 1$.
\end{proposition}

\subsection{Proof of Theorem \ref{Theorem1}} Firstly, we will prove the following results.
\vspace{0.2cm}

\begin{lemma}\label{lemma1.3}
    Under the assumptions of Theorem \ref{Theorem1}, the following estimates hold for all $u \in X(T,\varepsilon_0)$ and $T >0, \,\, \varepsilon_0 \in (0, 1]$:
    \begin{align*}
    \|\mathcal{N}(u(\tau,\cdot))\|_{L^{\gamma}} &\lesssim \mu\left(\varepsilon_0 (1+\tau)^{-\frac{n}{2}}\right) (1+\tau)^{-\frac{n}{2}(1+\frac{2}{n}-\frac{1}{\gamma})}  \|u\|_{X(T)}^{1+\frac{2}{n}} \quad\text{ for } \gamma \geq 1,\\
    \|\nabla \mathcal{N}(u(\tau,\cdot))\|_{L^{\theta}} &\lesssim \mu\left(\varepsilon_0 (1+\tau)^{-\frac{n}{2}}\right) (1+\tau)^{-\frac{n}{2}(1+\frac{2}{n}-\frac{1}{\theta})-\frac{1}{2}}  \|u\|_{X(T)}^{1+\frac{2}{n}} \quad \text{ for } \theta \in \left[2, \frac{2n}{[n-2]^+}\right).
    \end{align*}
\end{lemma}
\begin{proof}
    Firstly, applying interpolation to have
    \begin{align*}
        \|\mathcal{N}(u(\tau,\cdot))\|_{L^{\gamma}} &\leq \mu(\|u(\tau,\cdot)\|_{L^\infty})\||u(\tau,\cdot)|^{1+\frac{2}{n}}\|_{L^\gamma}  = \mu(\|u(\tau,\cdot)\|_{L^\infty})\|u(\tau,\cdot)\|_{L^{\gamma(1+\frac{2}{n})}}^{1+\frac{2}{n}} \\
        &\lesssim \mu\left((1+\tau)^{-\frac{n}{2}} \|u\|_{X(T)}\right) \|u(\tau,\cdot)\|_{L^{\alpha}}^{\frac{\alpha}{\gamma}} \|u(\tau,\cdot)\|_{L^{\infty}}^{1+\frac{2}{n}-\frac{\alpha}{\gamma}}  \\
        &\lesssim \mu\left(\varepsilon_0 (1+\tau)^{-\frac{n}{2}}\right)(1+\tau)^{-\frac{n}{2}(1+\frac{2}{n}-\frac{1}{\gamma})}  \|u\|_{X(T)}^{1+\frac{2}{n}}.
    \end{align*}
   Moreover, using the condition (\ref{Condition1.1.2}) yields
   \begin{align*}
       |\nabla \mathcal{N}(u(\tau, \cdot))| \lesssim \mu(|u(\tau,\cdot)|) |u(\tau,\cdot)|^{\frac{2}{n}} |\nabla u(\tau,\cdot)|.
   \end{align*}
   Therefore, we obtain the following relation:
   \begin{align*}
       \|\nabla \mathcal{N}(u(\tau,\cdot))\|_{L^\theta} \lesssim \|\mu(|u(\tau,\cdot)|)\|_{L^\infty} \|u(\tau, \cdot)\|_{L^\infty}^{\frac{2}{n}} \|\nabla u(\tau,\cdot)\|_{L^{\theta}}.
   \end{align*}
   Thanks to Proposition \ref{fractionalGagliardoNirenberg}, we arrive at
   \begin{align*}
       \|\nabla u(\tau,\cdot)\|_{L^\theta} &\lesssim \|u(\tau,\cdot)\|_{L^{\alpha}}^{1-\omega_1} \|u(\tau,\cdot)\|_{\dot{H}^2}^{\omega_1}\\
       &\lesssim (1+\tau)^{-\frac{n}{2}(1-\frac{1}{\theta})-\frac{1}{2}} \|u\|_{X(T)},
   \end{align*}
   where 
   \begin{align*}
       \omega_1 := \displaystyle\frac{\displaystyle\frac{1}{\alpha}-\frac{1}{\theta}+\frac{1}{n}}{\displaystyle\frac{1}{\alpha}-\frac{1}{2}+\frac{2}{n}} \in \left[\frac{1}{2}, 1\right]  \quad\text{ for all } \quad \theta \in \left[2, \frac{2n}{[n-2]^+}\right).
   \end{align*}
   From this, we get 
   \begin{align*}
       \|\nabla \mathcal{N}(u(\tau,\cdot))\|_{L^\theta} \lesssim \mu\left(\varepsilon_0 (1+\tau)^{-\frac{n}{2}}\right) (1+\tau)^{-\frac{n}{2}(1+\frac{2}{n}-\frac{1}{\theta})-\frac{1}{2}} \|u\|_{X(T)}^{1+\frac{2}{n}}.
   \end{align*}
   Thus, we completed the proof of Lemma \ref{lemma1.3}.
\end{proof}
\begin{proposition}\label{Pro1.1}
     Under the assumptions of Theorem \ref{Theorem1}, the following estimate holds for all $u \in X(T, \varepsilon_0),\, \varepsilon_0 \in (0, 1]$:
     \begin{align*}
         \|u^{\rm non}\|_{X(T)} \lesssim \mathcal{I}(\varepsilon_0)\|u\|_{X(T)}^{1+\frac{2}{n}}.
     \end{align*}
   
\end{proposition}
     \begin{proof}
         Firstly, using Lemma \ref{lemma1.3} and
         Lemma \ref{LinearEstimates} with $ j= 0, \,\rho = 1, q = 2$, $(s_1, s_2) =(2,0)$ for $\tau \in [0, t/2]$, and $\rho=q =2, (s_1,s_2) =(2,1)$ for $\tau \in (t/2, t]$ to derive
         \begin{align*}
             \| u^{\rm non}(t,\cdot)\|_{\dot{H}^2} &\lesssim \int_0^{t/2} (1+t-\tau)^{-\frac{n}{4}-1} \| \mathcal{N}(u(\tau,\cdot))\|_{L^1 \cap \dot{H}^1} d\tau\\
             &\qquad + \int_{t/2}^t (1+t-\tau)^{-\frac{1}{2}} \| \mathcal{N}(u(\tau,\cdot))\|_{ \dot{H}^1} d\tau\\
             &\lesssim \|u\|_{X(T)}^{1+\frac{2}{n}}\int_0^{t/2} (1+t-\tau)^{-\frac{n}{4}-1} (1+\tau)^{-1} \mu\left(\varepsilon_0 (1+\tau)^{-\frac{n}{2}}\right) d\tau\\
             &\qquad + \|u\|_{X(T)}^{1+\frac{2}{n}}\int_{t/2}^t (1+t-\tau)^{-\frac{1}{2}} (1+\tau)^{-\frac{n}{4}-\frac{3}{2}} \mu\left(\varepsilon_0 (1+\tau)^{-\frac{n}{2}}\right) d\tau\\
             &\lesssim  (1+t)^{-\frac{n}{4}-1}\|u\|_{X(T)}^{1+\frac{2}{n}} \int_0^{t/2}(1+\tau)^{-1} \mu\left(\varepsilon_0 (1+\tau)^{-\frac{n}{2}}\right) d\tau\\
             &\qquad +  (1+t)^{-\frac{n}{4}-1}\|u\|_{X(T)}^{1+\frac{2}{n}} \int_{t/2}^t (1+t-\tau)^{-1} \mu\left(\varepsilon_0 (1+t-\tau)^{-\frac{n}{2}}\right) d\tau,
         \end{align*}
where noting that $1+t \sim 1+\tau \geq 1+t-\tau$ for all $\tau \in [t/2, t]$. Using the change of variables $\theta_1 = \varepsilon_0 (1+\tau)^{-\frac{n}{2}}$ and $\theta_2 = \varepsilon_0(1+t-\tau)^{-\frac{n}{2}}$, we get
\begin{align}
    \int_0^{t/2} (1+\tau)^{-1} \mu\left(\varepsilon_0 (1+\tau)^{-\frac{n}{2}}\right) d\tau   +
    \int_{t/2}^{t} (1+t-\tau)^{-1} \mu\left(\varepsilon_0 (1+t-\tau)^{-\frac{n}{2}}\right) d\tau \lesssim \mathcal{I}(\varepsilon_0). \label{Impor_Re1}
\end{align}
Therefore, we can conclude that 
    \begin{align}
        \|u^{\rm non}(t,\cdot)\|_{\dot{H}^2} \lesssim \mathcal{I}(\varepsilon_0)(1+t)^{-\frac{n}{4}-1} \|u\|_{X(T)}^{1+\frac{2}{n}}. \label{Main.Es.1}
    \end{align}
 Next, thanks to again 
         Lemma \ref{LinearEstimates} for $j =0, \,\rho = 1, q = \infty$, $(s_1, s_2) = (0,0)$ for $\tau \in [0, t/2]$, and $\rho =r, q= \infty, \,(s_1, s_2) =(0,0)$ for $\tau \in (t/2, t]$ to obtain
    \begin{align*}
        \|u^{\rm non}(t,\cdot)\|_{L^\infty} &\lesssim \int_0^{t/2} (1+t-\tau)^{-\frac{n}{2}} \|\mathcal{N}(u(\tau,\cdot))\|_{L^1 \cap H_r^{d+\beta_r-1}} d\tau\\
        &\qquad +\int_{t/2}^t (1+t-\tau)^{-\frac{n}{2r}} \|\mathcal{N}(u(\tau,\cdot))\|_{L^r \cap H_r^{d+\beta_r-1}} d\tau,
    \end{align*}
    where $d := n/r +\delta$ with $\delta$ is a sufficiently small constant. 
Noting that the conditions $r \in \left(2, \displaystyle\frac{2n}{[n-2]^+}\right)$ and $1 \leq n \leq 4$ imply
\begin{align}
    0 < d + \beta_r = \frac{n}{r} +\delta + (n-1)\left(\frac{1}{2}-\frac{1}{r}\right) < 2,  \label{Re10}
\end{align}
   that is, 
   \begin{align*}
       \| \mathcal{N}(u(\tau,\cdot))\|_{H_r^{d+\beta_r-1}} &\lesssim \| \mathcal{N}(u(\tau,\cdot))\|_{H_r^1} \sim \|\mathcal{N}(u(\tau,\cdot))\|_{L^r} + \|\nabla \mathcal{N}(u(\tau,\cdot))\|_{L^r}\\ 
       &\lesssim (1+\tau)^{-\frac{n}{2}(1+\frac{2}{n}-\frac{1}{r})} \mu\left(\varepsilon_0 (1+\tau)^{-\frac{n}{2}}\right) \|u\|_{X(T)}^{1+\frac{2}{n}},
   \end{align*}
   due to using again Lemma \ref{lemma1.3}. Combining this with the relation (\ref{Impor_Re1}) and $n < 2r$, we get
   \begin{align}
       \|u^{\rm non}(t,\cdot)\|_{L^\infty} &\lesssim \|u\|_{X(T)}^{1+\frac{2}{n}}\int_0^{t/2} (1+t-\tau)^{-\frac{n}{2}} (1+\tau)^{-1} \mu\left(\varepsilon_0(1+\tau)^{-\frac{n}{2}}\right) d\tau \notag\\
       &\quad+ \|u\|_{X(T)}^{1+\frac{2}{n}} \int_{t/2}^t (1+t-\tau)^{-\frac{n}{2r}} (1+\tau)^{-\frac{n}{2}(1+\frac{2}{n}-\frac{1}{r})} \mu\left(\varepsilon_0 (1+\tau)^{-\frac{n}{2}}\right) d\tau \notag\\
       &\lesssim (1+t)^{-\frac{n}{2}} \|u\|_{X(T)}^{1+\frac{2}{n}} \int_0^{t/2} (1+\tau)^{-1} \mu\left(\varepsilon_0 (1+\tau)^{-\frac{n}{2}}\right) d\tau \notag\\
       &\quad + (1+t)^{-\frac{n}{2}} \|u\|_{X(T)}^{1+\frac{2}{n}} \int_{t/2}^t (1+t-\tau)^{-\frac{n}{2r}} (1+\tau)^{-1+\frac{n}{2r}}  \mu\left(\varepsilon_0 (1+\tau)^{-\frac{n}{2}}\right) d\tau \notag\\
       &\lesssim \mathcal{I}(\varepsilon_0) (1+t)^{-\frac{n}{2}} \|u\|_{X(T)}^{1+\frac{2}{n}}. \label{Main.Es.2}
   \end{align}
   Finally, we need to estimate $\|u^{\rm non}(\tau,\cdot)\|_{L^{\alpha}}$. We have $\beta_\alpha < 1$ for all $1 \leq n \leq 4$. Therefore, employing again Lemma \ref{lemma1.3} combined with Lemma \ref{LinearEstimates} with $j = 0, \,\rho = 1, q =\alpha, (s_1, s_2) = (0,0)$ for $\tau \in [0, t/2)$ and $\rho = q =\alpha, (s_1, s_2) = (0,0)$ for $\tau \in [t/2, t]$ we have
   \begin{align}
       \|u^{\rm non}(t,\cdot)\|_{L^\alpha} &\lesssim \int_0^{t/2} (1+t-\tau)^{-\frac{n}{2}(1-\frac{1}{\alpha})} \|\mathcal{N}(u(\tau,\cdot))\|_{L^1 \cap L^\alpha} d\tau  +\int_{t/2}^t \|\mathcal{N}(u(\tau,\cdot))\|_{L^{\alpha}} d\tau \notag\\
       &\lesssim \|u\|_{X(T)}^{1+\frac{2}{n}} \int_0^{t/2} (1+t-\tau)^{-\frac{n}{2}(1-\frac{1}{\alpha})} (1+\tau)^{-1} \mu\left(\varepsilon_0 (1+\tau)^{-\frac{n}{2}}\right) d\tau \notag\\
       &\quad+ \|u\|_{X(T)}^{1+\frac{2}{n}}\int_{t/2}^t (1+\tau)^{-\frac{n}{2}(1+\frac{2}{n}-\frac{1}{\alpha})} \mu\left(\varepsilon_0 (1+\tau)^{-\frac{n}{2}}\right) d\tau \notag\\
       &\lesssim \mathcal{I}(\varepsilon_0) (1+t)^{-\frac{n}{2}(1-\frac{1}{\alpha})} \|u\|_{X(T)}^{1+\frac{2}{n}}. \label{Main.Es.3}
   \end{align}
   The estimates (\ref{Main.Es.1})-(\ref{Main.Es.3}) complete the proof of Proposition \ref{Pro1.1}.
     \end{proof}

\begin{proposition}\label{Pro1.2}
    Under the assumptions of Theorem \ref{Theorem1}, the following estimate holds for all $u, v \in X(T, \varepsilon_0)$, $\varepsilon_0 \in (0, 1]$:
    \begin{align*}
        \|\Phi[u]-\Phi[v]\|_{Y(T)} \lesssim \mathcal{I}(\varepsilon_0) \|u-v\|_{Y(T)}\left(\|u\|_{X(T)}^{p-1} + \|v\|_{X(T)}^{p-1}\right).
    \end{align*}
\end{proposition}
    \begin{proof}
        From the definition of the spaces $Y(T)$ combined with Lemma \ref{LinearEstimates} with $j =0, \,\rho = m \in (1,2), q =2 $ and $(s_1,s_2)=(0,0)$, we obtain
        \begin{align*}
            &\|\Phi[u]-\Phi[v]\|_{Y(T)}\\
        &\quad\lesssim \sup_{t \in [0,T]} \left\{\int_0^t (1+t-\tau)^{-\frac{n}{2}(\frac{1}{m}-\frac{1}{2})}\|\mathcal{N}(u(\tau,\cdot))-\mathcal{N}(v(\tau,\cdot))\|_{L^m \cap H^{-1}} d\tau\right\}\\
            &\quad\lesssim \sup_{t \in [0,T]} \left\{\int_0^t (1+t-\tau)^{-\frac{n}{2}(\frac{1}{m}-\frac{1}{2})}\|\mathcal{N}(u(\tau,\cdot))-\mathcal{N}(v(\tau,\cdot))\|_{L^m} d\tau\right\},
        \end{align*}
        where we have the embedding
        \begin{align*}
            \|\varphi\|_{L^2} \lesssim \|\langle \nabla \rangle \varphi\|_{L^m} \quad\text{ with } \quad \frac{1}{m} \leq \frac{1}{2} +\frac{1}{n} \text{ and } m \in (1,2).
        \end{align*}
    Additionally,  using the assumption (\ref{Condition1.1.2}) we gain
    \begin{align*}
        &\left|\mathcal{N}(u(\tau,\cdot))-\mathcal{N}(v(\tau,\cdot))\right| \\
        &\quad=\left| (u-v)(\tau,\cdot) \times \left(\int_0^1  (\partial\mathcal{N})(v +\kappa(u-v)) d\kappa\right)(\tau,\cdot)  \right| \\
        &\quad\lesssim |(u-v)(\tau,\cdot)| \times \left|\left(\int_0^1 |v+ \kappa(u-v)|^{\frac{2}{n}} \mu\left(|v+\kappa(u-v)|\right) d\kappa\right)(\tau,\cdot)\right|\\
        &\quad\lesssim |(u-v)(\tau,\cdot) | \left(|u(\tau,\cdot)|^{\frac{2}{n}} + |v(\tau,\cdot)|^{\frac{2}{n}} \right) \int_0^1\|\mu(v+\kappa(u-v))(\tau,\cdot)\|_{L^\infty} d\kappa.
    \end{align*}
   Therefore,  thanks to H\"older's inequality and interpolation  one has
    \begin{align*}
        &\|\mathcal{N}(u(\tau,\cdot))-\mathcal{N}(v(\tau,\cdot))\|_{L^m} \\
        &\quad
        \lesssim \|u(\tau,\cdot)-v(\tau,\cdot)\|_{L^2} \left(\|u(\tau,\cdot)\|_{L^{\frac{2}{n} \eta}}^{\frac{2}{n}} + \|v(\tau,\cdot)\|_{L^{\frac{2}{n}\eta}}^{\frac{2}{n}}\right) \int_0^1 \|\mu(v+\kappa(u-v))(\tau,\cdot)\|_{L^\infty} d\kappa \\
        &\quad\lesssim (1+\tau)^{-1+\frac{n}{2\eta}} \mu\left(\varepsilon_0 (1+\tau)^{-\frac{n}{2}}\right) \|u-v\|_{Y(T)} \left(\|u\|_{X(T)}^{\frac{2}{n}}+\|v\|_{X(T)}^{\frac{2}{n}}\right)\\
        &\quad\lesssim (1+\tau)^{-1+\frac{n}{2}(\frac{1}{m}-\frac{1}{2})} \mu\left(\varepsilon_0 (1+\tau)^{-\frac{n}{2}}\right) \|u-v\|_{Y(T)}\left(\|u\|_{X(T)}^{\frac{2}{n}} + \|v\|_{X(T)}^{\frac{2}{n}}\right),
    \end{align*}
    provided that the following conditions are satisfied:
    \begin{align*}
     m \in (1,2),\quad  \frac{1}{m}=\frac{1}{2}+\frac{1}{\eta} \leq \frac{1}{2}+\frac{1}{n} \quad\text{ and } \alpha \leq \frac{2\eta}{n} < \infty.
    \end{align*}
    Thus, one can choose $1/ \eta =\delta \ll 1$ to ensure the existence of the parameter $m$.
    Summarily, we obtain the following estimate:
    \begin{align*}
        \|\Phi[u]-\Phi[v]\|_{Y(T)} \lesssim \sup_{t \in [0,T]} \big\{\mathcal{J}_1(t) +\mathcal{J}_2(t)\big\} \|u-v\|_{Y(T)}\left(\|u\|_{X(T)}^{\frac{2}{n}} + \|v\|_{X(T)}^{\frac{2}{n
        }}\right),
    \end{align*}
    where 
    \begin{align*}
        \mathcal{J}_1(t) &:=  \int_0^{t/2} (1+t-\tau)^{-\frac{n}{2}(\frac{1}{m}-\frac{1}{2})}(1+\tau)^{-1+\frac{n}{2}(\frac{1}{m}-\frac{1}{2})} \mu\left(\varepsilon_0 (1+\tau)^{-\frac{n}{2}}\right) d\tau\\
         &\lesssim (1+t)^{-\frac{n}{2}(\frac{1}{m}-\frac{1}{2})} \int_0^{t/2} (1+\tau)^{-1+\frac{n}{2}(\frac{1}{m}-\frac{1}{2})} \mu\left(\varepsilon_0 (1+\tau)^{-\frac{n}{2}}\right) d\tau \lesssim \mathcal{I}(\varepsilon_0),\\
         \mathcal{J}_2(t) &:= \int_{t/2}^t (1+t-\tau)^{-\frac{n}{2}(\frac{1}{m}-\frac{1}{2})}(1+\tau)^{-1+\frac{n}{2}(\frac{1}{m}-\frac{1}{2})} \mu\left(\varepsilon_0 (1+\tau)^{-\frac{n}{2}}\right) d\tau \\
         &\lesssim \int_{t/2}^{t} (1+t-\tau)^{-1} \mu\left(\varepsilon_0 (1+t-\tau)^{-\frac{n}{2}}\right) d\tau \lesssim \mathcal{I}(\varepsilon_0),
    \end{align*}
    for all $t > 0$ due to $n\left(1/m-1/2\right) < 2$ and $1 +\tau \geq 1+t-\tau$ for $\tau \in [t/2, t]$.
    Thus, we complete the proof of Proposition \ref{Pro1.2}.
 \end{proof}
    \textbf{Proof of Theorem \ref{Theorem1}.}  
Using again Lemma \ref{LinearEstimates} and the relation (\ref{Re10}) for the norm $L^\infty$, one derives
\begin{align*}
    \| u^{\rm lin}\|_{X(T)} \lesssim \varepsilon\|(u_0, u_1)\|_{\mathcal{D}},
\end{align*}
for all $T > 0$.
For this reason, combining with Propositions \ref{Pro1.1} and \ref{Pro1.2}, we obtain the following estimates for all $u,v \in X(T, \varepsilon_0), \,\varepsilon_0 \in (0, 1]$:
    \begin{align}
        \| \Phi[u]\|_{X(T)} &\leq C_1 \varepsilon \|(u_0, u_1)\|_{\mathcal{D}} +  C_1 \mathcal{I}(\varepsilon_0)\|u\|_{X(T)}^{1+\frac{2}{n}}, \label{Es.Pro2.1}\\
        \|\Phi[u]-\Phi[v]\|_{Y(T)} &\leq C_2 \mathcal{I}(\varepsilon_0)\|u-v\|_{Y(T)}\left(\|u\|_{X(T)}^{\frac{2}{n}}+\|v\|_{X(T)}^{\frac{2}{n}}\right).\label{Es.Pro2.2}
    \end{align}
    Let us fix 
    \begin{align}\label{Cons1}
    \bar{\varepsilon} := \frac{\varepsilon_0}{2 C_1\|(u_0, u_1)\|_{\mathcal{D}}}
    \end{align}
    and $\varepsilon_0$ satisfies
    \begin{align*}
        \max\{C_1, 2C_2\} \mathcal{I}(\varepsilon_0) \varepsilon_0^{\frac{2}{n}} \leq \frac{1}{2}.
    \end{align*}
    As a result, the inequalities (\ref{Es.Pro2.1}) and (\ref{Es.Pro2.2}) become
    \begin{align}
        \|\Phi[u]\|_{X(T)} &\leq \varepsilon_0, \label{Main.ine.1}\\
        \|\Phi[u]-\Phi[v]\|_{Y(T)} &\leq \frac{1}{2} \|u-v\|_{Y(T)}, \label{Main.ine.2}
    \end{align}
    for all $u,v \in X(T, \varepsilon_0),\, \varepsilon \in (0, \bar{\varepsilon}]$.
   Next, taking the recurrence sequence $\{u_j\}_{j=0}^{\infty}$ with $u_0 = 0;\,  u_{j} = \Phi[u_{j-1}]$ for $j = 1,2,...$, we employ (\ref{Main.ine.1}) to conclude that $$ \{u_j\}_{j=0}^{\infty} \subset X(T,\varepsilon_0), $$ for all $\varepsilon \in (0, \bar{\varepsilon}]$. 
Moreover, using (\ref{Main.ine.2}) implies that  $\{u_j\}_{j=0}^\infty$ is a Cauchy sequence in the Banach
space $Y(T)$.  Therefore, from Lemma \ref{lemma1.2}, there exists a solution $u = \Phi[u]$ in $X(T, \varepsilon_0)$ for all $T >0$. Since $T$ is arbitrary, the solution is global, that is, $u \in X(\infty, \varepsilon_0)$. 
 
 Next, we need to show that $u \in \mathcal{C}([0,\infty), H^2 \cap L^{\alpha} \cap L^{\infty})$. To prove this property, we recall the solution formula (\ref{Solution}).
 Since the linear part $u^{\rm lin}(t,x)$ obviously satisfies continuity, it suffices to show that
\begin{align}
    \int_0^t  \mathcal{K}(t-\tau,x) \ast_x \mathcal{N}(u(\tau,x)) d\tau \in \mathcal{C}([0,\infty), H^s \cap L^{\alpha} \cap L^\infty). \label{EQ2}
\end{align}
Noting that $u \in X(\infty, \varepsilon_0)$ and using again Lemmas \ref{LinearEstimates} and \ref{lemma1.3}, we have the following estimates:
\begin{align*}
    \|\mathcal{K}(t-\tau,x) \ast_x \mathcal{N}(u(\tau, x))\|_{L^{\alpha}} &\lesssim \mu\left(\varepsilon_0 (1+\tau)^{-\frac{n}{2}}\right) (1+\tau)^{-\frac{n}{2}(1+\frac{2}{n}-\frac{1}{\alpha})} \|u\|_{X(\infty)}^{1+\frac{2}{n}},\\
    \|\mathcal{K}(t-\tau,x) \ast_x \mathcal{N}(u(\tau, x))\|_{\dot{H}^2} &\lesssim \mu\left(\varepsilon_0 (1+\tau)^{-\frac{n}{2}}\right)  (1+\tau)^{-\frac{n}{4}-\frac{3}{2}} \|u\|_{X(\infty)}^{1+\frac{2}{n}}, \\
    \|\mathcal{K}(t-\tau,x) \ast_x \mathcal{N}(u(\tau, x))\|_{L^\infty} &\lesssim \mu\left(\varepsilon_0 (1+\tau)^{-\frac{n}{2}}\right) (1+\tau)^{-\frac{n}{2}-1} \|u\|_{X(\infty)}^{1+\frac{2}{n}}.
\end{align*}
Therefore, the Lebesgue convergence theorem in the Bochner integral immediately implies (\ref{EQ2}).

 Finally, we establish the uniqueness of the global solution 
$u$, which lies in the space $\mathcal{C}([0, \infty), H^2 \cap L^\alpha \cap L^\infty)$. Let $u, v$ be solutions to (\ref{Main.Eq.1}) belonging to this space. For the arbitrary positive number $T$, we can see that there exists a constant $M(T)$ such that $\|u\|_{X(T)}^{2/n}+ \|v\|_{X(T)}^{2/n} \leq M(T)$. Performing the same proof steps of Proposition \ref{Pro1.2}, we obtain
\begin{align*}
    \|u-v\|_{Y(t)} \lesssim M(T) \int_0^t \| u -v\|_{Y(\tau)}  d\tau,
\end{align*}
for all $t \in [0,T]$. 
From this, the
Gronwall inequality implies $ u \equiv v$ on $[0,T]$. Because $T $ is an arbitrary positive number, $u \equiv v$ on $[0, \infty)$.
Hence, the proof of Theorem \ref{Theorem1} is completed.
%...........................................................
\section{Asymptotic behaviors of global solution}\label{Section_Asym}
In this section, our main aim is to prove Theorem \ref{Theorem3}. To do this, we recall the following important results.
\vspace{0.2cm}

\begin{lemma}\label{Lemma3.1}
    Let $1 \leq \rho \leq q < \infty$, $q \ne 1$, $\beta_q := (n-1)\big|\frac{1}{2}-\frac{1}{q}\big|$ and $s_1 \geq s_2 \geq 0$. Then, the following estimate holds for $t \geq 1$: 
    \begin{align*}
        &\||\nabla|^{s_1}(\mathcal{K}(t,x)- \mathcal{G}(t,x)) \ast_x \varphi(x)\|_{L^q}\\ &\hspace{1cm}\lesssim (1+t)^{-\frac{n}{2}(\frac{1}{\rho}-\frac{1}{q})-\frac{s_1-s_2}{2}-1} \||\nabla|^{s_2} \chi_{\rm L}(|\nabla|)\varphi\|_{L^{\rho}} + e^{-ct} \||\nabla|^{s_1} \chi_{\rm H}(|\nabla|) \varphi\|_{H^{\beta_q-1}_q},
    \end{align*}
    where $c$ is a suitable positive constant. Furthermore, we obtain the following estimate for $ 1\leq \rho \leq \infty, 1 < r < \infty$ and $d > \displaystyle\frac{n}{r}$:
    \begin{align*}
        &\||\nabla|^{s_1}(\mathcal{K}(t,x)- \mathcal{G}(t,x)) \ast_x \varphi(x)\|_{L^\infty}\\ &\hspace{1cm}\lesssim (1+t)^{-\frac{n}{2\rho}-\frac{s_1-s_2}{2}-1} \||\nabla|^{s_2} \chi_{\rm L}(|\nabla|)\varphi\|_{L^{\rho}} + e^{-ct} \||\nabla|^{s_1+d} \chi_{\rm H}(|\nabla|) \varphi\|_{H^{\beta_r-1}_r}.
    \end{align*}
\end{lemma}
\begin{proof}
  The asymptotic behaviors in the norm $L^q$ with $1 < q< \infty$
 are obtained from Theorem 1.2 in \cite{Ikeda2019}. Therefore, we need to prove the estimate of the $L^\infty$ norm.  Proposition 2.10 in \cite{Ikeda2019} implies
    \begin{align*}
        &\||\nabla|^{s_1}\big(\chi_{\rm L}(|\nabla|)\mathcal{K}(t,x)- \mathcal{G}(t,x)\big) \ast_x \varphi(x)\|_{L^\infty} \\
        &\quad\quad\lesssim (1+t)^{-\frac{n}{2\rho}-\frac{s_1-s_2}{2}-1} \||\nabla|^{s_2} \chi_{\rm L}(|\nabla|)\varphi\|_{L^{\rho}}, 
    \end{align*}
    for all $1 \leq \rho \leq \infty$ and $s_1 \geq s_2 \geq 0$. Combining with the estimate (\ref{Esti100}), we can conclude the asymptotic behaviors in the norm $L^\infty$.
\end{proof}

\begin{lemma}\label{Lemma3.2}
    Let $1 \leq q \leq \infty,\,\, \gamma \geq 0,\, j \in \mathbb{N}$. Then, we have the following estimates for all $t > 0$:
        \begin{align}
        \|\partial_t^j |\nabla|^\gamma \mathcal{G}(t,\cdot)\|_{L^q} \lesssim t^{-\frac{n}{2}(1-\frac{1}{q})-\frac{\gamma}{2}-j}. \label{Es.of.lemm3.2}
    \end{align}
    Moreover, the following relation holds for $\gamma \in \mathbb{N}$ and $\varphi \in L^1, \,t \gg 1$:
    \begin{align}
        \left\||\nabla|^{\gamma}\mathcal{G}(t,x) \ast_x \varphi(x)- \left(\int_{\mathbb{R}^n} \varphi(x) dx\right) |\nabla|^{\gamma} \mathcal{G}(t,x)\right\|_{L^q} = o(t^{-\frac{n}{2}(1-\frac{1}{q})-\frac{\gamma}{2}}). \label{Re.of.lemma3.2}
    \end{align}
\end{lemma}

\begin{proof}
To prove the estimate (\ref{Es.of.lemm3.2}), we apply Theorem 24.2.4 in \cite{EbertReissig2018} as follows:
    \begin{align*}
        \|\partial_t^j |\nabla|^{\gamma} \mathcal{G}(t,\cdot)\|_{L^q} =C \left\|\mathfrak{F}^{-1}\left(|\xi|^{2j+\gamma}e^{-|\xi|^2t}\right)(t,\cdot)\right\|_{L^q} \lesssim t^{-\frac{n}{2}(1-\frac{1}{q})-\frac{\gamma}{2}-j}.
    \end{align*}
  The relation (\ref{Re.of.lemma3.2}) follows directly from Lemma A.1 in \cite{IkehataMichihisa2019}.
\end{proof}
\textbf{Proof of Theorem \ref{Theorem3}.} To begin with, we recall that the global solution to (\ref{Main.Eq.1}) satisfies the following integral equation (\ref{Solution}). Next,
we shall use the notation and results established in Section \ref{Section2} and fix $(\theta, s)$ as one of the two pairs $(\alpha, 0)$ and $(2,2)$. Lemmas \ref{Lemma3.1} and \ref{Lemma3.2} yield
\begin{align*}
    &\left\||\nabla|^s \left(u^{\rm lin}(t,x)-\varepsilon\left(\int_{\mathbb{R}^n}(u_0(x) + u_1(x))dx\right)\mathcal{G}(t,x)\right) \right\|_{L^\theta}\\
    &\quad \lesssim \varepsilon\left\||\nabla|^s\big(\mathcal{K}(t,x)-\mathcal{G}(t,x)\big)\ast_x (u_0+ u_1)(x)\right\|_{L^\theta} \\
    &\quad\qquad\quad+ \varepsilon\left\||\nabla|^s\left(\mathcal{G}(t,x) \ast_x(u_0+u_1)(x)-\left(\int_{\mathbb{R}^n} (u_0(x)+u_1(x))dx\right) \mathcal{G}(t,x)\right) \right\|_{L^\theta} \\
    &\quad\qquad\quad+ \varepsilon\|\partial_t |\nabla|^s \mathcal{K}(t,x) \ast_x u_0(x)\|_{L^\theta} \\
    &\quad = o(t^{-\frac{n}{2}(1-\frac{1}{\theta})-\frac{s}{2}}), \quad t \gg 1.
\end{align*}
Here we used Lemma \ref{Lemma3.1} for the first term, the relation (\ref{Re.of.lemma3.2}) for the second term, and Lemma \ref{LinearEstimates} for the third term.
Therefore, we need to prove the following relation:
\begin{align}
    &\left\||\nabla|^s \left(u^{\rm non}(t,x)- \left(\int_0^{\infty} \int_{\mathbb{R}^n} \mathcal{N}(u(t, x))dxdt\right) \mathcal{G}(t,x) \right)\right\|_{L^{\theta}} \notag\\
    &\hspace{4cm}= o(t^{-\frac{n}{2}(1-\frac{1}{\theta})-\frac{s}{2}}) , \label{Main.Rela.1}
\end{align}
for $(\theta, s) \in \{(2,2), (\alpha, 0)\}$ and $ t \gg 1$.  Let us now divide the left-hand side term of (\ref{Main.Rela.1}) in the norm into five parts as follows:
\begin{align}
    &u^{\rm non}(t,x)- \left(\int_0^{\infty} \int_{\mathbb{R}^n} \mathcal{N}(u(\tau, x))dxd\tau\right) \mathcal{G}(t,x) \notag\\
    &\quad= \int_0^t \mathcal{K}(t-\tau, x) \ast_x \mathcal{N}(u(\tau,x)) d\tau - \left(\int_0^{\infty} \int_{\mathbb{R}^n} \mathcal{N}(u(\tau, x))dxd\tau\right) \mathcal{G}(t,x) \notag\\
    &\quad= \int_0^{t/2} (\mathcal{K}(t-\tau, x)-\mathcal{G}(t-\tau,x)) \ast_x \mathcal{N}(u(\tau,x)) d\tau \notag\\
    &\quad\quad+ \int_{t/2}^t \mathcal{K}(t-\tau, x) \ast_x \mathcal{N}(u(\tau,x)) d\tau \notag\\
    &\quad\quad+ \int_0^{t/2} (\mathcal{G}(t-\tau,x) - \mathcal{G}(t,x)) \ast_x \mathcal{N}(u(\tau,x)) d\tau\notag\\
    &\quad\quad + \int_0^{t/2} \mathcal{G}(t,x) \ast_x \mathcal{N}(u(\tau,x)) - \left(\int_{\mathbb{R}^n} \mathcal{N}(u(\tau, x))dx\right) \mathcal{G}(t,x) d\tau\notag\\
    &\quad\quad - \left(\int_{t/2}^\infty \int_{\mathbb{R}^n} \mathcal{N}(u(\tau, x))dxd\tau\right) \mathcal{G}(t,x)\notag\\
    &\quad =: K_1(t,x)+ K_2(t,x) + K_3(t,x) + K_4(t,x) - K_5(t,x). \label{Relation0}
\end{align}
We proved that Theorem \ref{Theorem1} ensures the existence of a unique global solution $u$ satisfying $$\|u\|_{X(\infty)} \leq \varepsilon_0.$$ 
Therefore, thanks to again Lemma \ref{lemma1.3} and Lemma \ref{Lemma3.1} for $\rho = 1, \, q = \theta$ and $(s_1, s_2) = (s,0)$, we obtain
\begin{align}
    \||\nabla|^s K_1(t,\cdot)\|_{L^\theta} &\lesssim \int_0^{t/2} (1+t-\tau)^{-\frac{n}{2}(1-\frac{1}{\theta})-\frac{s}{2}-1} \|\mathcal{N}(u(\tau,\cdot))\|_{L^1 \cap H^{s+\beta_\theta-1}_{\theta}} d\tau \notag\\
    & \lesssim \varepsilon_0^{1+\frac{2}{n}} (1+t)^{-\frac{n}{2}(1-\frac{1}{\theta})-\frac{s}{2}-1} \int_0^{t/2} (1+\tau)^{-1} \mu\big(\varepsilon_0 (1+\tau)^{-\frac{n}{2}}\big) d\tau\notag\\
    &\lesssim \mathcal{I}(\varepsilon_0) \varepsilon_0^{1+\frac{2}{n}} (1+t)^{-\frac{n}{2}(1-\frac{1}{\theta})-\frac{s}{2}-1} = o\big(t^{-\frac{n}{2}(1-\frac{1}{\theta})-\frac{s}{2}}),  \quad t \gg 1, \label{Relation1}
\end{align}
for $(\theta, s) = (\alpha, 0)$ or $(2,2)$ and $ t \gg 1$. Now, we consider the term $K_2(t,x)$. Specicially, employing again Lemma \ref{lemma1.3} and Lemma \ref{LinearEstimates} for $\rho =q =\theta$ and $(s_1, s_2) = (s, s/2)$ we get
\begin{align}
    &\||\nabla|^s K_2(t,\cdot)\|_{L^\theta}\notag\\ 
    &\quad\lesssim \int_{t/2}^t (1+t-\tau)^{-\frac{s}{4}}\|\chi_{\rm L}(|\nabla|)\mathcal{N}(u(\tau,\cdot))\|_{\dot{H}^{s/2}_\theta} \notag\\
    &\hspace{3cm}+ e^{-c(t-\tau)} \|\chi_{\rm H}(|\nabla|)\mathcal{N}(u(\tau,\cdot))\|_{\dot{H}^{s+\beta_\theta-1}_{\theta}} d\tau\notag\\
    &\quad\lesssim  \varepsilon_0^{1+\frac{2}{n}}\int_{t/2}^t (1+t-\tau)^{-\frac{s}{4}}(1+\tau)^{-\frac{n}{2}(1+\frac{2}{n}-\frac{1}{\theta})-\frac{s}{4}} \mu\left(\varepsilon_0 (1+\tau)^{-\frac{n}{2}}\right) d\tau\notag\\
    &\quad\lesssim \varepsilon_0^{1+\frac{2}{n}} (1+t)^{1-\frac{n}{2}(1+\frac{2}{n}-\frac{1}{\theta})-\frac{s}{2}} \mu(\varepsilon_0 (1+t/2)^{-\frac{n}{2}}) \notag\\
    &\quad= \varepsilon_0^{1+\frac{2}{n}} (1+t)^{-\frac{n}{2}(1-\frac{1}{\theta})-\frac{s}{2}} \mu(\varepsilon_0 (1+t/2)^{-\frac{n}{2}}) = o(t^{-\frac{n}{2}(1-\frac{1}{\theta})-\frac{s}{2}}), \quad t \gg 1,\label{Relation2}
\end{align}
 where we note that $\mu(0) = 0$, $\mu$ is continuity and $\beta_\alpha < 1, \, \beta_2 = 0$. Taking account of $K_3(t,x)$ we use the mean value theorem on $t$ to get the following representation:
\begin{align*}
    \mathcal{G}(t-\tau,x) - \mathcal{G}(t,x) = -\tau \partial_t\mathcal{G}(t- \lambda_1 \tau, x),
\end{align*}
with $\lambda_1=\lambda_1(t,\tau)\in [0,1]$. Thus, using the relation $t -\lambda_1 \tau \sim t$ for $\tau \in [0, t/2]$ and the estimate (\ref{Es.of.lemm3.2}),  the  following estimates hold for $t \gg 1$:
\begin{align}
    &\||\nabla|^s K_3(t,\cdot)\|_{L^\theta}\notag\\ 
    &\quad\lesssim \int_0^{t/2} \tau(t- \lambda_1\tau)^{-\frac{n}{2}(1-\frac{1}{\theta})-\frac{s}{2}-1} \|\mathcal{N}(u(\tau,\cdot))\|_{L^1} d\tau\notag\\
    &\quad\lesssim \varepsilon_0^{1+\frac{2}{n}}t^{-\frac{n}{2}(1-\frac{1}{\theta})-\frac{s}{2}-1} \int_0^{t/2} \tau (1+\tau)^{-1} \mu\big(\varepsilon_0 (1+\tau)^{-\frac{n}{2}} \big) d\tau\notag\\
    &\quad\lesssim \varepsilon_0^{1+\frac{2}{n}}t^{-\frac{n}{2}(1-\frac{1}{\theta})-\frac{s}{2}-1} \int_0^{t/2} \mu\big(\varepsilon_0 (1+\tau)^{-\frac{n}{2}} \big) d\tau = o(t^{-\frac{n}{2}(1-\frac{1}{\theta})-\frac{s}{2}}),  \label{Relation3}
\end{align}
due to applying the L'Hospital rule, we get
\begin{align*}
    \lim_{t \to \infty} t^{-1} \int_0^{t/2} \mu\big(\varepsilon_0 (1+\tau)^{-\frac{n}{2}} \big) d\tau = 0.
\end{align*}
Next, we estimate the quantity $K_4(t,x)$ by dividing two parts as follows:
\begin{align*}
    K_4(t,x) &= \int_0^{t/2} \int_{\mathbb{R}^n} \big(\mathcal{G}(t, x-y)-\mathcal{G}(t,x)\big)\mathcal{N}(u(\tau,y)) dy d\tau \\
    &= \int_0^{t/2} \int_{|y| \leq t^{1/4}} \big(\mathcal{G}(t, x-y)-\mathcal{G}(t,x)\big)\mathcal{N}(u(\tau,y)) dy d\tau\\
    &\qquad + \int_0^{t/2} \int_{|y| \geq t^{1/4}} \big(\mathcal{G}(t, x-y)-\mathcal{G}(t,x)\big)\mathcal{N}(u(\tau,y)) dy d\tau\\
    & =: K_{41}(t,x) + K_{42}(t,x).
\end{align*}
For the term $K_{41}(t,x)$, we can present the following formula:
\begin{align*}
    \mathcal{G}(t, x-y) - \mathcal{G}(t,x) = \int_0^1 y\nabla \mathcal{G}(t,x-\eta y) d\eta.
\end{align*}
From this, employing Lemma \ref{Lemma3.1} to have
\begin{align*}
    \||\nabla|^s K_{41}(t,\cdot)\|_{L^\theta} &\lesssim \int_0^1\int_0^{t/2} \int_{|y| \leq t^{1/4}} |y|\||\nabla|^{s+1} \mathcal{G}(t, x-\eta y)\|_{L^\theta} |\mathcal{N}(u(\tau,y))| dy d\tau d\eta\\
    &\lesssim t^{-\frac{n}{2}(1-\frac{1}{\theta})-\frac{s+1}{2}} \int_0^{t/2} t^{\frac{1}{4}}\|\mathcal{N}(u(\tau,\cdot))\|_{L^1} d\tau\\
    &\lesssim \varepsilon_0^{1+\frac{2}{n}} t^{-\frac{n}{2}(1-\frac{1}{\theta})-\frac{s}{2}-\frac{1}{4}} \int_0^{t/2} (1+\tau)^{-1} \mu\big(\varepsilon_0 (1+\tau)^{-\frac{n}{2}}\big) d\tau \\
    &= o(t^{-\frac{n}{2}(1-\frac{1}{\theta})-\frac{s}{2}}), \quad t \gg 1.
\end{align*}
Next, employing a change of variable $\lambda := \varepsilon_0 (1+\tau)^{-\frac{n}{2}}$, we gain
\begin{align*}
    \int_0^{\infty}\int_{\mathbb{R}^n} \mathcal{N}(u(\tau,y))dy d\tau &= \int_0^{\infty} \|\mathcal{N}(u(\tau,\cdot))\|_{L^1} d\tau \\
    &\lesssim \int_0^{\infty} (1+\tau)^{-1} \mu\big(\varepsilon_0 (1+\tau)^{-\frac{n}{2}}\big) d\tau = C \int_0^{\varepsilon_0}\frac{\mu(\lambda)}{\lambda} d\lambda < \infty.
\end{align*}
This immediately implies
\begin{align*}
    \lim_{t \to \infty}\int_0^{t/2}\int_{|y| \geq t^{1/4}}   \mathcal{N}(u(\tau,y)) dyd\tau = 0.
\end{align*}
Therefore, thanks to (\ref{Es.of.lemm3.2}), the term $K_{42}(t,x)$ will be estimated as follows:
\begin{align*}
    \||\nabla|^s K_{42}(t,\cdot)\|_{L^\theta} &\lesssim \int_0^{t/2} \int_{|y| \geq t^{1/4}} \||\nabla|^s(\mathcal{G}(t, x-y)-\mathcal{G}(t,x))\|_{L^\theta
    } \mathcal{N}(u(\tau,y)) dyd\tau\\
    &\lesssim t^{-\frac{n}{2}(1-\frac{1}{\theta})-\frac{s}{2}} \int_0^{t/2}\int_{|y| \geq t^{1/4}}  \mathcal{N}(u(\tau,y)) dyd\tau = o(t^{-\frac{n}{2}(1-\frac{1}{\theta})-\frac{s}{2}}), \quad t \gg 1.
\end{align*}
In summary, we gain
\begin{align}
    \||\nabla|^s K_4(t,\cdot)\|_{L^\theta} = o(t^{-\frac{n}{2}(1-\frac{1}{\theta})-\frac{s}{2}}), \quad t \gg 1. \label{Relation4}
\end{align}
Finally, we can control $K_5(t,x)$ as follows:
\begin{align}
    \||\nabla|^s K_5(t,\cdot)\|_{L^\theta} &\lesssim \||\nabla|^s\mathcal{G}(t,\cdot)\|_{L^\theta} \int_{t/2}^\infty \|\mathcal{N}(u(\tau,\cdot))\|_{L^1} d\tau\notag\\
    &\lesssim \varepsilon_0^{1+\frac{2}{n}} t^{-\frac{n}{2}(1-\frac{1}{\theta})-\frac{s}{2}} \int_{t/2}^{\infty} (1+\tau)^{-1} \mu\big(\varepsilon_0 (1+\tau)^{-\frac{n}{2}}\big) d\tau\notag\\
    &= o(t^{-\frac{n}{2}(1-\frac{1}{\theta})-\frac{s}{2}}), \quad t \gg 1, \label{Relation5}
\end{align}
due to
\begin{align*}
    \lim_{t \to \infty} \int_{t/2}^{\infty} (1+\tau)^{-1} \mu\big(\varepsilon_0 (1+\tau)^{-\frac{n}{2}}\big) d\tau = C \lim_{t \to \infty} \int_0^{\varepsilon_0(1+t/2)^{-n/2}} \frac{\mu(\lambda)}{\lambda} d\lambda = 0.
\end{align*}
The relation (\ref{Relation0})-(\ref{Relation5}) immediately imply (\ref{Main.Rela.1}). Therefore, we obtain the asymptotic behavior of the solution in the norms $L^\alpha$ and $\dot{H}^2$. As for the asymptotic behavior in the $L^\infty$ norm, some minor adjustments are needed due to the estimate in Lemma \ref{Lemma3.1}. Nevertheless, the proof proceeds in exactly the same way as in the $L^\alpha$ case and Theorem \ref{Theorem1}. Hence, we conclude the proof of Theorem \ref{Theorem3} here.
%...........................................................
\section{Some remarks to sharp lifespan  estimates}\label{Lifespan}
Following the argument in the proof of Theorem \ref{Theorem1}, one obtains the existence of  global (in time) solutions to the Cauchy problem \eqref{Main.Eq.1}. On the other hand, Theorem 5 in \cite{EbeGirRei2020} shows that the solution to \eqref{Main.Eq.1} blows up in finite time for all $n \geq 1$ when the non-Dini condition (\ref{Condition1.2.1}) occurs. Now, we define the following function: 
\begin{align}\label{defPsi}
    \Psi(R):=\displaystyle\int_{1}^{R}\frac{\mu\left(Cr^{-\frac{n}{2}}\right)}{r}dr
\end{align}
for a sufficiently large positive constant $C$. The sharp lifespan estimates for blow-up solutions are determined as follows.
\vspace{0.2cm}

\begin{proposition}[\textbf{Sharp lifespan estimates}]\label{Theorem2}
    Let us consider the problem (\ref{Main.Eq.1}) with the dimension spaces $1\leq n\leq 4$. The  modulus of continuity $\mu(s)$ satisfies the non-Dini condition (\ref{Condition1.2.1})     and the function
     \begin{align}\label{Condition1.2.3}
     \mathcal{N}: s \in \mathbb{R} \to \mathcal{N}(s) := |s|^{1+\frac{2}{n}} \mu(|s|) \in [0, \infty) \text{ is convex }.
     \end{align}
     In addition, we assume that the initial data $(u_0,u_1)\in \mathcal{D}$ satisfies
     \begin{align*}
         \int_{\mathbb{R}^n} (u_0(x)+ u_1(x)) dx > 0.
     \end{align*}
     Then, there exists a positive constant $\varepsilon_0$ such that for any $\varepsilon\in (0,\varepsilon_0]$, the solution $u$ blows up in finite time and the upper bound estimate for
the lifespans
\begin{align}
    T_\varepsilon \lesssim \Psi^{-1}\left(C\varepsilon^{-\frac{2}{n}}\right) \label{Upper_lifespan}
\end{align}
holds. Furthermore, if we replace the condition (\ref{Condition1.2.3}) by (\ref{Condition1.1.2}),
      the lifespan of solution $$u \in \mathcal{C}([0, T_\varepsilon), H^2 \cap L^\alpha \cap L^\infty) 
     $$
      satisfies the lower bound estimate
     \begin{align}
         T_{\varepsilon}\gtrsim \Psi^{-1}\left(c\varepsilon^{-\frac{2}{n}}\right), \label{Lower_Lifespan}
     \end{align}
     for any $\varepsilon\in (0,\varepsilon_0]$.
     Here $C$, $c$ are positive constants depending only on $n,u_0,u_1,\mu$ and $\Psi^{-1}(\tau)$ is the inverse function of the function $\Psi(\tau)$ defined in \eqref{defPsi}.
\end{proposition}
\vspace{0.2cm}

\begin{remark}
\fontshape{n}
\selectfont
 Linking the achieved estimates in Proposition \ref{Theorem2}, one recognizes that the lifespan $T_\varepsilon$ in spatial dimensions $1 \leq n \leq 4$ under the non-Dini condition are determined by the following relations:
 \begin{align*}
   \Psi  (T_{\varepsilon}) \sim \varepsilon^{-\frac{2}{n}}.
 \end{align*}
In fact, under the assumptions of Proposition \ref{Theorem2}, the authors in \cite{ChenGirardi2025} proved (\ref{Upper_lifespan}) for all $n \geq 1$ and (\ref{Lower_Lifespan}) for $n=1$. In addition, the approach in \cite{Girardi2024} can be easily applied to derive an estimate (\ref{Lower_Lifespan}) for $n=2,3$. Therefore, the novelty of Proposition \ref{Theorem2} lies in establishing estimate (\ref{Lower_Lifespan}) for $n=4$.

\end{remark}

 \textbf{Proof of the estimate (\ref{Lower_Lifespan})}.
The proof relies on the argument introduced in \cite{Ikeda2016}, together with the solution spaces and the quantities established in the proof of Theorem \ref{Theorem1}. Specifically, we recall some estimates related to the global existence of small data solutions.
$$
\|u^{\text{lin}}\|_{X(T)} \leq c_{0}^*\varepsilon\|(u_0,u_1)\|_{\mathcal{D}}.
$$
On the other hand, we see that 
\begin{equation*}
\begin{aligned}
\|u^{\text{non}}(t,\cdot)\|_{\dot{H}^2}&\lesssim  (1+t)^{-\frac{n}{4}-1}\|u\|_{X(T)}^{1+\frac{2}{n}} \int_0^{t/2}(1+\tau)^{-1} \mu\left(C(1+\tau)^{-\frac{n}{2}}\right) d\tau\\
&\qquad + (1+t)^{-\frac{n}{4}-1}\|u\|_{X(T)}^{1+\frac{2}{n}} \int_{t/2}^t (1+t-\tau)^{-1} \mu\left(C(1+t-\tau)^{-\frac{n}{2}}\right) d\tau\\
&\lesssim (1+t)^{-\frac{n}{4}-1}\Psi(1+T)\|u\|^{1+\frac{2}{n}}_{X(T)}.
\end{aligned}
\end{equation*}
Similarly, we also have
\begin{equation*}
\begin{aligned}
\|u^{\rm non}(t,\cdot)\|_{L^\infty}&\lesssim (1+t)^{-\frac{n}{2}}\Psi(1+T)\|u\|_{X(T)}^{1+\frac{2}{n}},\\
\|u^{\rm non}(t,\cdot)\|_{L^\alpha}&\lesssim (1+t)^{-\frac{n}{2}(1-\frac{1}{\alpha})}\Psi(1+T)\|u\|_{X(T)}^{1+\frac{2}{n}}.
\end{aligned}
\end{equation*}
As a result, we gain
\begin{equation*}
\|\Phi[u]\|_{X(T)}\leq \frac{c_0 \varepsilon}{2} +c_1\Psi(1+T)\|u\|^{1+\frac{2}{n}}_{X(T)},
\end{equation*}
where $c_0= 2c_0^*\|(u_0,u_1)\|_{\mathcal{D}}$. 
Similarly, we can immediately conclude that
\begin{align*}
    \|\Phi[u]-\Phi[v]\|_{Y(T)} \leq c_2 \Psi(1+T) \|u-v\|_{Y(T)} \left(\|u\|_{X(T)}^{\frac{2}{n}} + \|v\|_{X(T)}^{\frac{2}{n}}\right).
\end{align*}
Therefore, for all $u, v \in X(T, c_0\varepsilon)$, the last two estimates imply that
\begin{align*}
    \|\Phi[u]\|_{X(T)}  &\leq \frac{c_0 \varepsilon}{2} + c_1 \Psi(1+T) c_0 ^{1+\frac{2}{n}}\varepsilon^{1+\frac{2}{n}},\\
    \|\Phi[u]-\Phi[v]\|_{Y(T)} &\leq 2 c_2c_0^{\frac{2}{n}} \varepsilon^{\frac{2}{n}}\Psi(1+T)  \|u-v\|_{Y(T)},
\end{align*}
where constants $c_1$ and $c_2$ independent of $M, \varepsilon$ and $T$. Therefore, if we assume that
\begin{align*}
    \max\{c_1, 2c_1\} c_0^{\frac{2}{n}} \varepsilon^{\frac{2}{n}} \Psi(1+T) < \frac{1}{4},
\end{align*}
then we may construct a unique local solution $u \in X(T , c_0\varepsilon)$ by an argument entirely analogous to the proof of Theorem \ref{Theorem1}. Moreover, the following estimate
holds:
\begin{align}\label{Ine.1}
    \|u\|_{X(T)} = \|\Phi[u]\|_{X(T)} < \frac{3c_0\varepsilon}{4}.
\end{align}
Afterwards, we consider
$$
T^*:=\sup \left\{T \in\left[0, T_{\varepsilon}\right) \text { such that } F(T):=\|u\|_{X(T)} \leq c_0 \varepsilon\right\}. 
$$
If $T^*$ satisfies the inequality (\ref{Ine.1}), then from the previous estimate it follows that  $F(T^*) < 3c_0\varepsilon/4$. 
Since $F(T)$ is a continuous function, there exists $\overline{T}\in (T^*,T_{\varepsilon})$ satisfies $F(\overline{T})\leq c_0\varepsilon$, which contradicts the definition of $T^*$. Hence, we deduce that 
\[
\max\{c_1, 2c_2\}\Psi\left(T^*+1\right)c_0^{\frac{2}{n}}\varepsilon^{\frac{2}{n}}\geq \frac{1}{4},
\]
which leads to
$$
T_{\varepsilon} \geq T^* \gtrsim\Psi^{-1}\left(c\varepsilon^{-\frac{2}{n}}\right).
$$
This completes the proof of the lower bound for the lifespan $T_\varepsilon$.

%...............................................................................................

%...............................................................................................

\section*{Acknowledgment.}
This work is supported by Vietnam Ministry of Education and Training
and Vietnam Institute for Advanced Study in Mathematics under grant number B2026-CTT-04.  The authors would like to thank Prof. Tuan Anh Dao (Hanoi University of Science
and Technology), Prof. Wenhui Chen (Guangzhou University), and Prof. Giovanni Girardi (Università Politecnica delle Marche) for their helpful advice in the preparation of this
paper.

%=================================================================================={References}
\end{document}